\documentclass[12pt,oneside]{article}
\usepackage{geometry}
\usepackage{graphicx}
\usepackage{todonotes}
\usepackage{amsmath}
\usepackage{amsfonts}
\usepackage{amssymb}
\usepackage{tikz-cd}
\usepackage{comment}
\usepackage{bbm}
\usepackage{stmaryrd}
\usepackage{fullpage}
\usepackage{epstopdf}
\usepackage{amsthm}
\usepackage{mathrsfs} %this one is the really fancy mathscr
\usepackage{tikz}
\usepackage{dsfont}
\usetikzlibrary{matrix}
\usepackage{mathtools}
\usepackage[hidelinks]{hyperref}
\usepackage{cleveref}
\usepackage{xcolor}
\usepackage{enumitem}
\usepackage{tensor}
\usepackage[utf8]{inputenc}
\usepackage{csquotes}
\usepackage[english]{babel}
\usepackage{mymacros}

\hypersetup{
	colorlinks,%Colours links instead of ugly boxes
	linkcolor={red!50!black},%Colour of internal links
	citecolor={blue!50!black},%Colour of citations
	urlcolor={blue!80!black}%Colour for external hyperlinks
}

\tikzset{
	rot90/.style={anchor=south, rotate=90, inner sep=.5mm}
}
\tikzset{
	rot45/.style={anchor=south, rotate=-45, inner sep=.5mm}
}

\newtheorem{theorem}{Theorem}[section]

\newtheorem{innercustomthm}{Theorem}
\newenvironment{customthm}[1]
{\renewcommand\theinnercustomthm{#1}\innercustomthm}
{\endinnercustomthm}

\theoremstyle{definition}
\newtheorem{remark}[theorem]{Remark}%No italics
\title{Categorifying reduced rings}
\author{Ishan Levy\thanks{The author was supported by the NSF Graduate Research Fellowship under Grant No. 1745302.}}
\usepackage{setspace}
\usepackage{babel,csquotes,xpatch}
\usepackage[backend=bibtex,style=alphabetic,maxnames = 99,maxalphanames=99]{biblatex}
\makeatletter
\renewcommand\tableofcontents{%
	\@starttoc{toc}%
}
\makeatother

\newcommand{\ZCat}{\mathbb{Z}\mathrm{Cat}}

\newcounter{counter}
\newtheorem{thm}[counter]{Theorem}
\newtheorem{prop}[counter]{Proposition}
\newtheorem{qst}[counter]{Question}
\newtheorem{cnj}[counter]{Conjecture}
\newtheorem{dfn}[counter]{Definition}
\newtheorem{prb}[counter]{Problem}

\newtheorem{lem}[counter]{Lemma}
\newtheorem{rmk}[counter]{Remark}

\newcommand{\Addresses}{{% additional braces for segregating \footnotesize
		\bigskip
		\footnotesize
}}

\begin{document}
	\date{}
	\maketitle
	\begin{abstract}
		Given a domain of characteristic zero $R$, we functorially construct a rigid symmetric monoidal dg-category whose $K_0$ is $R$, solving a problem of Khovanov. We also functorially construct, for any reduced commutative ring $R$, a rigid braided monoidal dg-category whose $K_0$ is $R$.
	\end{abstract}
%	\begin{spacing}{0.1}
%		\tableofcontents
%	\end{spacing}

	One prevalent insight in mathematics is that many classical invariants admit categorifications. Namely, instead of assigning a number or polynomial to an object $X$, one assigns an object of some stable category\footnote{Throughout this paper, we use category to mean $\infty$-category in the sense of Joyal and Lurie (see \cite{HTT}, \cite{HA} or \cite{kerodon}). If the reader prefers, they may read the phrase `stable category' as `dg-category'. In our language, a dg-category over $k$ is a $k$-linear stable category \cite{cohn2013differential}. All stable categories of interest in this paper are $k$-linear for some commutative ring $k$.}
	%'triangulated category'. The homotopy $1$-category of a stable category is naturally triangulated \cite[Theorem 1.1.2.14]{HA}. However for some purposes in this paper, such as \Cref{thm:versal}, it is important that we work with stable categories
	 $C$, which encodes richer information about $X$. The original invariant can then be recovered via applying an Euler characteristic, i.e a homomorphism out of $K_0(C)$\footnote{$K_0(C)$, i.e the Grothendieck group of $C$, is the abelian group generated by $[c]$ for each $c \in C$ with the relation $[c]=[c']+[c'']$ whenever there is a cofibre sequence $c' \to c \to c''$.}, to the $K_0$ class of the categorified invariant.
	
	A prototypical example of this is in algebraic topology, where the Euler characteristic of a space is categorified by its homology. Another example is in algebraic geometry, where the Hilbert polynomial is categorified by coherent cohomology. Finally, in low dimensional topology, classical knot invariants such as the Jones polynomial and Alexander polynomial are categorified by Khovanov homology \cite{khovanov2000categorification} and Knot Floer homology \cite{ozsvath2004holomorphic} respectively. 
	
	A natural question to ask is whether the type of invariant can obstruct categorifications from existing.
	For example, the Witten--Reshetikhin--Turaev invariants, constructed in \cite{witten1989quantum,reshetikhin1991invariants}, are invariants taking values in the complex numbers, and it is not known whether these always admit a categorification. 
	
	One can build stable categories whose $K_0$ is an arbitrary rational vector space, as is done in \cite{barwick2019categorifying}. The reason is that the multiplication by $n$ map on $K_0$ can be implemented by a functor: for example the functor sending $c \mapsto \oplus_1^nc$. By taking a filtered colimit along such a functor, since $K_0$ preserves filtered colimits, one obtains a category whose $K_0$ is that of $C$, but with $n$ inverted.
	
	The method above doesn't naturally produce monoidal categories, which is something usually desired for the target categories of categorifications of invariants, to allow for K\"unn\-eth-type formulas for the invariants one is categorifying. This is also a fundamental feature of categorifications of manifold invariants that come from field theories. Along these lines, Khovanov posed the following problems:
	
	\begin{prb}{\cite[Problem 2.3]{khovanov2016linearization}}\label{prob1}
		Construct a stable\footnote{By a monoidal stable category we mean an algebra object in the category of categories that is stable, and such that the tensor product preserves finite colimits in both variables.} monoidal category $C$ with $K_0\cong\QQ$.
	\end{prb}

	\begin{prb}\cite[Problem 2.4]{khovanov2016linearization}\label{prob2}
		Construct a stable monoidal category $C$ with $K_0\cong \ZZ[\frac 1 n]$.
	\end{prb}

	\cite{khovanov2019categorify} solved \Cref{prob2} in the case $n=2$, but the general case was previously open.
	
	We solve both problems below in \Cref{cor:subring}, even producing categories that are rigid\footnote{We recall that a monoidal category is rigid if every object is strongly left and right dualizable.} symmetric monoidal. The fundamental input is the ability to produce a category whose $K_0$ is a characteristic zero field, which we construct as an ultraproduct of the stable module categories of $\FF_p[C_p]$. We further refine this in \Cref{thm:functorial} below. 
	
 	Let $\Ring_{\red}^0$ be the category of commutative rings which are reduced, and such that every minimal prime ideal is characteristic zero\footnote{For example, this includes all domains of characteristic zero.}. Let $\ZCat^{\infty}_{\rig}$ denote the category of rigid symmetric monoidal $\ZZ$-linear stable categories.
	
	\begin{customthm}{A}\label{thm:functorial}
		There is a filtered colimit preserving functor $C^{\infty}_{(-)}:\Ring_{\red}^0 \to \ZCat^{\infty}_{\rig}$ and a natural isomorphism of rings $K_0(C^{\infty}_R) \cong R$.
	\end{customthm}

	To illustrate the purpose of \Cref{thm:functorial}, $C^{\infty}_{\overline{\QQ}}$ is a rigid symmetric monoidal $\ZZ$-linear category with a continuous action of the absolute Galois group of $\QQ$, such that $K_0(C^{\infty}_{\overline{\QQ}}) \cong \overline{\QQ}$ with its standard action.
	
	We next remove the characteristic zero assumption of \Cref{thm:functorial} in the setting of braided monoidal categories. The key input here is the category of tilting modules in the mixed case for Lusztig's quantum group $\msl_2$, which we semisimplify and tensor with the stable module category of $\FF_p[C_p]$ in order to categorify finite fields. Let $\Ring_{\red}$ be the category of reduced commutative rings, and $\ZCat^{2}_{\rig}$ the category of rigid braided monoidal $\ZZ$-linear stable categories.
%	Combining our results with the work of Laugwitz--Qi on categorifying cyclotomic rings \cite{laugwitz2018categorification}, we remove the characteristic $0$ assumption of \Cref{thm:functorial} in the setting of monoidal categories. Let $\Ring_{\red}$ denote the category of reduced commutative rings, and $\Alg_{\EE_2}(\Cat_{\ZZ}^{\st})_{\rig}$ the category of rigid braided monoidal $\ZZ$-linear stable categories.
	
	\begin{customthm}{B}\label{thm:functorialmon}
		There is a filtered colimit preserving functor $C^2_{(-)}:\Ring_{\red} \to \ZCat^2_{\rig}$ and a natural isomorphism of rings $K_0(C^2_R) \cong R$.
	\end{customthm}

	There is a natural condition on a category for which our constructions are essentially sharp.
	
	\begin{dfn}
		A braided monoidal category $C$ is \text{trace-zero} if for every nilpotent endomorphism $f:c \to c$, the trace of $f$ is zero. 
	\end{dfn}

	Our constructions of braided monoidal categories as described above gives the following result\footnote{Note that the dimension gives a ring homomorphism $K_0(C_R) \to \End(\unit_{C_R})$ to endomorphism ring of the unit. In particular, $\End(\unit_{C_R})$ is characteristic $p$ when $p=0$ in $R$.}:

%	For example, any spherical monoidal abelian category is trace-zero, as are their bounded derived categories and Verdier quotients thereof [cite]. Our constructions more precisely give:
	
	\begin{customthm}{C}\label{thm:trzero}
		For every reduced commutative ring $R$, there is a $\ZZ$-linear ribbon braided monoidal trace-zero stable category $C_R$ with an isomorphism of rings $K_0(C_R) \cong R$.
	\end{customthm}

	The following result is an obstruction to producing trace-zero categorifications of rings that are more symmetric than those in \Cref{thm:trzero}:
%	\begin{lem}[{\cite[Theorem 3.15]{etingof2021lectures}}]\label{thm:obstruction}
%		Let $C$ be a symmetric monoidal trace-zero stable category with $p =0$ in $\End(\unit)$. Then the imageo f the dimension homomorphism $K_0(\FF_p) \to \End_0(\unit)$ is $\FF_p$.
%	\end{lem}

	\begin{thm}[{\cite[Theorem 5.15]{etingof2021lectures}}]\label{thm:obstruction}
			Let $C$ be a ribbon braided monoidal trace-zero stable\footnote{The cited result is for additive $1$-categories rather than stable categories, but the relevant part of the assumptions only depend the homotopy $1$-category of $C$ which is additive, so this is ok.} category with $[\unit_C,\unit_C]$ a characteristic $p$ ring. If $v\in \Aut(\unit_C)$ is the twist automorphism, suppose that $v^\ell$ is unipotent for some $\ell$ coprime to $p$. Let $n$ be the smallest integer such that $p^n-1$ is divisible by $\ell^2$. Then for any object $V \in C$, we have $\dim V \in \FF_{p^n}$. In particular, $K_0(C)$ cannot contain any field extension of $\FF_{p^n}$.
		\end{thm}
%	There is more generally an obstruction for trace-zero ribbon categories with a unipotence condition on the balancing automorphism from having large finite fields in $K_0$ (see \cite[Theorem 5.15]{etingof2021lectures}), 
	In particular, for trace-zero symmetric monoidal categories, the dimension homomorphism $K_0(C) \to \End(\unit_C)$ must land inside $\FF_p$ (see \cite[Lemma 3.13]{etingof2021lectures}), so $K_0(C)$ cannot contain any field extension of $\FF_p$. 
	
	Despite the above obstruction, there are a number of possible directions for improvement to our theorems. Notably the categories of \Cref{thm:functorial} are not idempotent-complete, so one could ask whether it is possible to build idempotent-complete categories. Another possible improvement to \Cref{thm:functorial} would be an extension to all commutative rings, or at least the removal of the assumption about the minimal primes being characteristic zero. The method of proof of \Cref{thm:functorial} works without the characteristic zero assumption given a positive solution to the following conjecture:

	\begin{cnj}\label{conj:fpbar}
	For each prime $p$, there exists a rigid symmetric monoidal $\ZZ$-linear category $C$ with $K_0(C) \cong \overline{\FF}_p$.
	\end{cnj}

	By \Cref{thm:obstruction}, any positive solution to \Cref{conj:fpbar} must not be trace-zero.

		\subsection*{Acknowledgements}
	I am very grateful to Nitu Kitchloo for introducing me to \Cref{prob2}, and also to Mikhail Khovanov for posing the problem. I am also very grateful to Pavel Etingof for pointing out to me that $\msl_2$-tilting modules in the mixed case could be used to give ribbon categorifications of finite fields. I would like to thank Pavel Etingof, David Gepner, Mikhail Khovanov, Nitu Kitchloo, and Vijay Srinivasan for discussions related to this problem. Finally, I would like to thank Pavel Etingof, Mikhail Khovanov, and Nitu Kitchloo for helpful feedback on earlier drafts.
	
\section{The examples}
	In this section we exhibit symmetric monoidal categories whose $K_0$ is an arbitrary domain of characteristic zero. Combining this with categories built from $\msl_2$-tilting modules in the mixed case, we produce ribbon categories whose $K_0$ is an arbitrary reduced ring.
	
	 We first introduce some categories and operations that we use. Consider the cate\-go\-r\-y $\StMod^{\omega}_{C_p}$\-\footnote{The category $\StMod^{\omega}_{C_p}$ can also be described as the category of perfect modules over the $\EE_{\infty}$-ring $\FF_p^{tC_p}$, the ring whose homotopy ring is the Tate cohomology of $\FF_p$ with a trivial $C_p$-action.}, the compact objects of the stable module category of the group ring $\FF_p[C_p]$ (see for example \cite[Section 2]{mathew2015torus} for an overview of this category). This is an idempotent-complete rigid symmetric monoidal $\ZZ$-linear category, and $K_0(\StMod^{\omega}_{C_p}) \cong \FF_p$.
	
	The construction in \Cref{thm:k0char0} below uses an ultraproduct of the categories $\StMod^{\omega}_{C_p}$ over the set $\PP$ of primes. We briefly recall how ultraproducts work, referring the reader to \cite[Section 3]{barthel2020chromatic} or \cite[Section 3.6]{etingof2021lectures} for more details. Given a family $C_{\alpha}, \alpha \in A$ of objects in a category $D$ ($D$ could be the category of categories), we can choose a non-principal ultrafilter $U$ on the set $A$, which is the data of a maximal proper filter in the poset of subsets of $A$ containing all cofinite subsets. Then an ultraproduct\footnote{The ultraproduct depends on the choice of ultrafilter, but we disregard this in our notation.}, which we denote $\Pi'_{\alpha \in A}C_{\alpha}$ is the object of $D$ defined as the filtered colimit along the subsets of $S$ in $U$ of the product $\Pi_{s \in S}C_{s}$. An ultrapower of an object $C$ is an ultraproduct where all of the $C_{\alpha}$ are the same object, $C$.
	
	\begin{lem}\label{thm:k0char0}
		$\Pi'_{p \in \PP} \StMod^{\omega}_{C_p}$ is an idempotent-complete rigid symmetric monoidal $\ZZ$-linear category with $K_0$ a field of characteristic zero. It is possible to choose an ultrafilter so that $K_0$ contains $\overline{\QQ}$.
	\end{lem}

	\begin{proof}
		The functor $K_0$ preserves arbitrary products
		%\footnote{In fact, this is true for the universal filtered colimit preserving additive invariant \cite[Theorem 1.4]{kasprowski2019algebraic}.}
		 and filtered colimits, so it preserves ultraproducts. Since $K_0(\StMod^{\omega}_{C_p})$ is $\FF_p$, the result follows, as an ultraproduct of fields of characteristic $p$ for different $p$ is one of characteristic $0$. It follows from \cite[Theorem 5]{ax1967solving} that one can choose an ultrafilter so that $K_0$ contains $\overline{\QQ}$.
	\end{proof}

\begin{rmk}\label{rmk:constructive}
	We chose a non-principal ultrafilter on an infinite set in the proof of \Cref{thm:k0char0} which is a non-constructive operation, but it is possible to do the proof in a constructive way, with the cost of the category not being idempotent-complete. Namely, one can simply choose a filter on the set of primes such that the corresponding colimit of products of prime fields still contains $\overline{\QQ}$ as in the proof of \cite[Theorem 5]{ax1967solving}. By then considering the subcategory of objects whose $K_0$ class lives in $\overline{\QQ}$, we obtain an explicit construction of a category with $K_0$ isomorphic to $\overline{\QQ}$.
\end{rmk}
%The following lemma is well known (see for example ).
%
%\begin{lem}\label{lem:ultrafilter}
%	There exists an ultrafilter on the set $\PP$ of primes such that the associated ultraproduct $\Pi'_{p \in \PP}\FF_p$ contains $\overline{\QQ}$.
%\end{lem}
%
%\begin{proof}
%	Choose a countable increasing family of finite Galois extensions of $\QQ$, $F_n$, such that $\colim_n F_n = \overline{\QQ}$. By the Chebotarev density theorem\footnote{The fact that we need is really more elementary than the Chebotarev density theorem: namely that any nonconstant polynomial with integer coefficients has a root modulo infinitely many primes.}, the set of primes $P_n$ completely splitting in $F_n$ is infinite for each $n$. For each $F_n$, choose a primitive integral element $\alpha_n$, and let $f_n$ be its minimal polynomial. Because the order generated by $\alpha_n$ agrees with the ring of integers for almost all primes, $f_n$ has a root for all but finitely many primes in $P_n$.
%	
%	Choose an ultrafilter containing $P_n$ for each $n$, which is possible since $P_n \supset P_{n+1}$. In the ultraproduct of $\FF_p$ associated to such an ultrafilter, $f_n$ has a root, since it has a root in all but finitely many primes in $P_n$. Thus the ultraproduct is a characteristic zero field containing each $F_n$, so it contains $\overline{\QQ}$.
%\end{proof}
\begin{dfn}
	Given a rigid braided monoidal category $C$, and an object $x \in C$, we can form the dual $x^{*}$, giving an equivalence $(-)^*:C \cong C^{\mathrm{op}}$. We refer to this as the \textit{dual functor}.
\end{dfn}

\begin{thm}\label{cor:subring}
	For any reduced commutative ring $R$ such that every minimal prime ideal of $R$ is characteristic $0$, there is a rigid symmetric monoidal $\ZZ$-linear category $C$ such that $K_0(C) \cong R$, and the dual functor is the identity on $K_0$.
\end{thm}
\begin{proof}
	We claim that the set of commutative rings $R$ satisfying the theorem is closed under products, ultraproducts, and subrings. The claim for products and ultraproducts follows since $K_0$ commutes with products and ultraproducts. For subrings, if $K_0(C) \cong R$, $R' \subset R$ is a subring, and the dual is the identity on $K_0(C)$, then the full subcategory of $C$ consisting of objects whose $K_0$ class is in $R'$ is a category showing that the result is true for $R'$.
	
	Any reduced ring embeds into the product of its localizations at all of its minimal primes. Indeed this is because the kernel of a ring mapping into all of its residue fields is exactly its nilradical, and the kernel of any residue field contains the kernel of some minimal prime ideal. In the case of a minimal prime, the localization at the prime is exactly the residue field associated to that prime.
	
	By assumption, each of these localizations is a field of characteristic zero, so we have thus reduced the theorem to the case $R$ is a characteristic zero field. An arbitrary field of characteristic zero embeds into an algebraically closed field, and ultrapowers of $\overline{\QQ}$ give examples of algebraically closed fields of arbitrarily large transcendence degree, so it suffices to produce a rigid symmetric monoidal $\ZZ$-linear category with the dual acting by the identity such that $K_0$ contains $\overline{\QQ}$. Now we apply \Cref{thm:k0char0}, and conclude by observing that the dual functor induces the identity on $K_0(\StMod^{\omega}_{C_p})$.
	
%	We first note that by taking ultrapowers of the category from \Cref{thm:k0char0} with $K_0$ containing $\overline{\QQ}$, we can produce rigid symmetric monoidal stable categories with $K_0(C)$ containing an arbitrarily large field of characteristic $0$. Any $R$ reduced embeds into the product of its residue fields at its minimal primes, which we have assumed here to be characteristic $0$ rings, so it follows that the product of these residue fields embeds into $K_0(C')$, where $C'$ is a product of ultrapowers of categories coming from \Cref{thm:k0char0}.
%	
%	Next, we observe that the functor $C' \to (C')^{op}$ given by taking duals is the identity on $K_0$: this follows from the fact that it is true for $\StMod^{\omega}_{C_p}$, and $C'$ was constructed from these via products and filtered colimits. It follows that taking $C$ to be the subcategory of $C'$ whose $K_0$ is contained in $R \subset K_0(C')$, $C$ is a rigid symmetric monoidal stable category, and has the desired $K_0$.
\end{proof}

%	\begin{cor}\label{cor:subring}
%		For any subring $R \subset \overline{\QQ}$, there exists a rigid symmetric monoidal stable category $C$ with $K_0(C) \cong R$.
%	\end{cor}
%	\begin{proof}
%		Let $C'$ be the category of \Cref{thm:k0char0}, so that $K_0(C')$ is a field of characteristic zero containing $\overline{\QQ}$.
%		First note that the functor $C' \to (C')^{op}$ given by taking duals is the identity on $K_0$: this follows from the fact that it is true for $\StMod^{\omega}_{C_p}$. It follows that for any subring $R \subset K_0(C')$, the subcategory $C$ of objects whose $K_0$ is contained in $R$ is a rigid symmetric monoidal stable category. Since $\overline{\QQ} \subset K_0(C')$ we can in particular get any subring of $\overline{\QQ}$.
%	\end{proof}

%We next explain how using \cite{laugwitz2018categorification}, it is possible to construct rigid monoidal stable categories whose $K_0$ contains $\overline{\FF}_p$. 

We next explain an improvement of \Cref{cor:subring} in the setting of ribbon braided monoidal categories.
The key example of a ribbon braided monoidal category that we use is constructed in the following proposition:

\begin{prop}\label{prop:sl2ribbon}
	For every pair of distinct primes $p, \ell>2$, there is an idempotent-complete semisimple $\overline{\FF}_p$-linear ribbon category $C$ with $K_0(C)\cong \ZZ[\zeta_\ell+\zeta_\ell^{-1}]$.
\end{prop}

\begin{proof}
	Let $q$ be a primitive $\ell^{th}$ root of unity in $\overline{\FF}_p$. We consider the category $\Tilt^{\overline{\FF}_p,q}$ of tilting modules of Lusztig's divided power quantum group associated to $\msl_2$ over the field $\overline{\FF}_p$. We refer the reader to \cite{sl2tiltingmod} for a good reference for this category: it is an idempotent-complete $\overline{\FF}_p$-linear ribbon category with indecomposable objects $T(v)$ for $v\geq 0$, where $T(0)$ is the unit.
	
	Let $C'$ be the semisimplification of this category: see \cite{semisimplification} for an overview of this construction. $C'$ is then an idempotent-complete $\overline{\FF}_p$-linear semisimple ribbon category. We will show an appropriate subcategory $C \subset C'$ has $K_0$ isomorphic to $\ZZ[\zeta_{\ell}+\zeta_{\ell}^{-1}]$. The simple objects in the semisimplification $C'$ correspond to indecomposable objects in $\Tilt^{\overline{\FF}_p,q}$ with nonzero categorical dimension, which is exactly $T(i)$ for $0\leq i \leq \ell-2$ by \cite[Proposition 3.23]{sl2tiltingmod}. The tensor product by \cite[Lemma 4.1]{sl2tiltingmod} agrees with the truncated Clebsch--Gordon rule for the tensor product in the Verlinde category $\Ver_p$ \cite[Section 4.2]{etingof2021lectures}. 
	
	We let $C$ be the full subcategory of $C'$ generated by $T(i)$ for $i$ even: this is also an idempotent-complete $k$-linear semisimple ribbon category since the even objects are closed under tensor product, duals, and contain the unit. Then $K_0(C)$ then agrees with $K_0(\Ver_p^{+})$, which is indeed $\ZZ[\zeta_\ell+\zeta_\ell^{-1}]$ by \cite[Theorem 4.5]{incompressible}.
\end{proof}

\begin{remark}
	The key properties of the categories of \Cref{prop:sl2ribbon} that we use are that their mod $p$ reductions can contain arbitrarily large field extensions of $\FF_p$. If we were interested in constructing categories that are just monoidal as opposed to braided monoidal, there are other candidates, such as the categories constructed in \cite{laugwitz2018categorification}, whose $K_0$ is an arbitrary cyclotomic extension of the integers.
\end{remark}

The next goal is to realize the operation $K_0 \mapsto K_0/(p)$ at the level of categories. To do this, we use the tensor product of $\FF_p$-linear stable categories. If $C,D$ are $\FF_p$-linear stable categories, then $C\otimes_{\FF_p}D$ is also such a category that is generated as a stable category by objects $c\otimes d, c \in C, d \in D$, and $\map(c\otimes d,c'\otimes d') \cong \map(c,c')\otimes_{\FF_p}\map(d,d')$. In particular, $K_0(C\otimes_{\FF_p}D)$ has a surjective map from $K_0(C)\otimes K_0(D)$.

\begin{prop}\label{prop:catmodp}
	Let $C$ be an $\FF_p$-linear abelian category. Then if $D^{b}(C)$ is the bounded derived category of $C$, then $K_0(D^b(C)\otimes_{\FF_p} \StMod^{\omega}_{C_p}) \cong K_0(C)/(p)$.
\end{prop}
\begin{proof}
	Let $C' = D^b(C)\otimes_{\FF_p} \StMod^{\omega}_{C_p}$.
	There is a surjective map from $$K_0(D^b(C)) \otimes  K_0(\StMod^{\omega}_{C_p}) \to K_0(C')$$ By the Gillet-Waldhausen theorem, $K_0(D^b(C)) \cong K_0(C)$, and since $K_0(\StMod^{\omega}_{C_p}) \cong \FF_p$, the above map is really a surjective map $K_0(C)/(p) \to K_0(C')$. It suffices to show that this map is injective.
	
	Let $\Rep^{\omega}_{\FF_p}(C_p)$ denote the bounded derived category of finite dimensional representations of the cyclic group $C_p$ over the field $\FF_p$, so that $\StMod^{\omega}_{C_p}$ is the quotient of $\Rep^{\omega}_{\FF_p}(C_p)$ by the full stable subcategory generated by the projective representation, which is equivalent to $\Mod^{\omega}(\FF_p[C_p])$.
	
	Tensoring with $D^b(C)$, and applying the connective $K$-theory functor (see \cite{BGT}), we get a sequence
	
	\begin{center}
		\begin{tikzcd}
			K(D^b(C)\otimes_{\FF_p}\Mod^{\omega}(\FF_p[C_p]))\ar[r] &K(D^b(C)\otimes_{\FF_p}\Rep^{\omega}_{\FF_p}(C_p)) \ar[r] &  K(C')
		\end{tikzcd}
	\end{center}
	
	This is a cofibre sequence since connective $K$-theory sends localization sequences of categories to cofibre sequences if the quotient is surjective on $K_0$. 

	Let $f: \Mod(\FF_p)^{\omega} \to \Rep_{\FF_p}^{\omega}(C_p)$ denote the functor giving an $\FF_p$-module the trivial $C_p$-action.
	It suffices to show that the natural map $$K_0(D^b(C)) \xrightarrow{K_0(D^b(C)\otimes f)} K_0(D^b(C)\otimes_{\FF_p}\Rep^{\omega}_{\FF_p}(C_p))$$ is an equivalence, and that the image of $K_0(D^b(C)\otimes_{\FF_p}\Mod^{\omega}(\FF_p[C_p]))$ under the second map is divisible by $p$, so that the kernel of the map $K_0(C)/(p) \to K_0(C')$ is contained in the ideal $(p)$ and hence is zero.
	
	The first claim follows from \Cref{lem:dev} below. For the second, we note that the functor $f$ has a retraction $g:\Rep^{\omega}_{\FF_p}(C_p) \to \Mod^{\omega}(\FF_p)$ given by taking the underlying $\FF_p$-module. $K_0(D^b(C)\otimes g)$ is an equivalence since it is an inverse of $K_0(D^b(C)\otimes f)$, which we have already seen is an isomorphism. Thus it suffices to show that the composite $$K_0(D^b(C)\otimes_{\FF_p}\Mod^{\omega}(\FF_p[C_p]) \to K_0(D^b(C)\otimes_{\FF_p}\Rep^{\omega}_{\FF_p}(C_p)) \xrightarrow{g} K_0(D^b(C)\otimes_{\FF_p}\Mod^{\omega}(\FF_p))$$ has image in the ideal $(p)$. This is because this functor has a filtration given by the filtration of $\FF_p[C_p]$ by the powers of the maximal ideal, whose associated graded is a direct sum of $p$ copies of the functor given by basechange along the map of rings $\FF_p[C_p] \to \FF_p$.
\end{proof}

\begin{lem}\label{lem:dev}
	Let $C$ be an $\FF_p$-linear category with bounded $t$-structure. Then the natural map $K(C) \xrightarrow{K(C\otimes f)} K(C\otimes_{\FF_p}\Rep_{\FF_p}^{\omega}(C_p))$ is an equivalence, and $K_{-1}$ of both categories vanish.
\end{lem}
\begin{proof}
	To prove the lemma, we will show that $C\otimes f$ satisfies the conditions of \cite[Theorem 1.3]{kcoconn}. The image of $f$ generates the target since each finite dimensional representation of $C_p$ over $\FF_p$ always has a nonzero fixed vector. It thus suffices to show that the functor is fully faithful on the heart. If $a,b\in C^{\heart}$, then $\pi_*\map(C\otimes f(a), C\otimes f(b)) \cong \pi_*\map(a,b)\otimes_{\FF_p}H^{-*}(C_p;\FF_p)$. Since $\map(a,b)$ is coconnective and $H^{-*}(C_p;\FF_p)$ is concentrated in nonpositive degrees with $H^0 \cong \FF_p$, it follows that $\pi_0\map(a,b) \cong \pi_0\map(C\otimes f(a),C\otimes f(b))$, i.e that the functor is fully faithful.
\end{proof}

Trace-zero categories are closed under the operations we use:

\begin{lem}\label{lem:trzeroquot}
	If $C$ is a rigid braided monoidal abelian category, then $C$ is trace-zero. Any rigid braided monoidal stable category with bounded $t$-structure is also trace-zero. Trace-zero stable categories are also closed under idempotent completion, products, and filtered colimits.
\end{lem}

\begin{proof}
	The fact that $C$ is trace-zero is well-known: given a nilpotent endomorphism $f:c \to c$, one can filter $c$ by the kernel of the powers of $f$ to find that the map $f$ is zero on the associated graded. Thus the trace of $f$ is zero since it is on the associated graded.
	
	To see that a rigid braided monoidal category with bounded $t$-structure is trace-zero, we use the fact that every map is canonically filtered by the Postnikov tower, with associated graded maps in shifts of the heart. In particular, using additivity of traces, if $f$ is a nilpotent endomorphism of $X$, it suffices to check for each $i$ that the map $\pi_{i}^{\heart}f:\pi_i^{\heart}X \to \pi_i^{\heart}X$ has trace zero. Each $\pi_i^{\heart}f$ is nilpotent since $f$ is.
	
	The heart of the $t$-structure is an abelian category, and so the proof of the previous case applies to $\pi_i^{\heart}f$.\footnote{Note that this doesn't directly reduce to the previous case, since the dual of an object in the heart may not be in the heart.}
	 
	 We omit the proof that trace-zero categories are closed under idempotent completion, products, and filtered colimits, as it is straightforward.
%	 Finally, let $C'$ be a quotient of a trace-zero stable category $C$ by some tensor ideal $I$, and let $f:c \to c$ be a nilpotent endomorphism in $C'$. Because $f$ is nilpotent, there is a lift $\tilde{f}:\tilde{c} \to \tilde{c}$ in $f^n$ factors through some object $d \in I$ as a map $g:c \to d$ and a map $h:d \to c$. By cyclicity of the trace, the trace of $f^n$ is the same as the trace of 
\end{proof}

We are now ready to produce ribbon categorifications of all reduced rings.

\begin{thm}\label{cor:charp}
	For any reduced ring $R$, there exists a trace-zero ribbon braided monoidal $\ZZ$-linear stable category $C$ with $K_0(C)\cong R$ such that the dual is the identity on $K_0$.
\end{thm}

\begin{proof}
	As in \Cref{cor:subring}, The set of $R$ satisfying the theorem are closed under subrings and ultraproducts. The collection of reduced rings is generated under subrings and ultraproducts by the collection of finite fields $\FF_{p^n}$.
	
	Let $C$ be a tensor product (relative to $\overline{\FF}_p$) of categories coming from \Cref{prop:sl2ribbon}. Since these categories are semisimple over an algebraically closed field, $C$ is also semisimple, and the $K_0$ of this tensor product is the tensor product of the $K_0$ of each factor. By \Cref{prop:catmodp}, $D^b(C)\otimes \StMod^{\omega}_{C_p}$ is a ribbon braided monoidal $\ZZ$-linear stable category with $K_0$ a tensor product of the mod $p$ reductions of $\ZZ[\zeta_\ell+\zeta_{\ell}^{-1}]$ for various $\ell>2$. To see it is trace-zero, we first observe that since $C$ is semisimple, the underlying stable category of $D^b(C)\otimes \StMod^{\omega}_{C_p}$ is a product of copies of $\overline{\FF}_p\otimes_{\FF_p}\StMod^{\omega}_{C_p}$ corresponding to the simple objects of $C$. Now the category is trace-zero for the same reason that $\StMod^{\omega}_{C_p}$ is, which we now explain. 
	
	Given a nilpotent endomorphism $f$ of an object $x \in D^b(C)\otimes \StMod^{\omega}_{C_p}$, we can choose a lift $\tilde{f}$ of $f$ to an endomorphism of an object $\tilde{x}$ in the heart of $D^b(C)\otimes \Rep_{\FF_p}(C_p)$ containing no projective summand. Traces of endomorphisms factor through the semisimplification, which kills the projective representation. By assumption $\tilde{f}^n$ factors through some projective representation for sufficiently large $n$, so $\tilde{f}$ is nilpotent in the semisimplification, and thus has zero trace.
	
	Since $\FF_{p^n}$ is the tensor product of finite fields of the form $\FF_{p^{q^n}}$, it suffices to show that for each $p,q,n$ with $p,q$ prime, there is an $\ell>2$ such that $\FF_{p^{q^n}}$ is contained in the mod $p$ reduction of $\ZZ[\zeta_\ell]$. Indeed, to see that the claim for $\ZZ[\zeta_{\ell}]$ implies the one for $\ZZ[\zeta_\ell+\zeta_{\ell}^{-1}]$, the claim for $\ZZ[\zeta_\ell]$ is equivalent to the claim that $\FF_{p^{q^n}}$ is contained in all the characteristic $p$ residue fields of $\ZZ[\zeta_\ell]$. Since $\ZZ[\zeta_\ell+\zeta_{\ell}^{-1}] \to \ZZ[\zeta_{\ell}]$ is a degree $2$ extension, it follows that in $\ZZ[\zeta_\ell+\zeta_{\ell}^{-1}]$, the residue fields mod $p$ contain a subextension of index at most $2$, and so in particular contain $\FF_{p^{q^{n-1}}}$.
	
	To see the claim for $\ZZ[\zeta_\ell]$, we first observe that whenever there is a prime $\ell$ dividing $|\FF_{p^{q^n}}|^\times$ that doesn't divide $|\FF_{p^{q^{n-1}}}|^\times$, then $\ZZ[\zeta_{\ell}]/p$ contains $\FF_{p^{q^n}}$.
	
	We thus need to show that for infinitely many values of $n$, $|\FF_{p^{q^n}}|^\times = p^{q^n}-1$ contains primes not dividing $p^{q^{n-1}}-1$. This will be true if $\gcd(\frac{p^{q^n}-1}{p^{q^{n-1}}-1},p^{q^{n-1}}-1) = \gcd(q,p^{q^{n-1}}-1)$\footnote{This equality follows by writing $\frac{p^{q^n}-1}{p^{q^{n-1}}-1} = 1+ \dots + p^{q^{n-1}} + \dots + p^{(q-1)q^{n-1}}$ and using that $p^{q^{n-1}} \equiv 1$ mod $(p^{q^{n-1}} - 1)$.} is $1$, i.e if $p$ is not congruent to $1$ mod $q$.
	
	If $q|p-1$, then for large $n$, the $q$-adic valuation of $p^{q^{n-1}}-1$ is larger than $1$. Thus since $\gcd(\frac{p^{q^n}-1}{p^{q^{n-1}}-1},p^{q^{n-1}}-1) = q$ by the above equality, it follows that for large $n$, the $q$-adic valuation of $\frac{p^{q^n}-1}{p^{q^{n-1}}-1}$ is $1$. Thus $\frac{p^{q^n}-1}{p^{q^{n-1}}-1}$ must have a prime factor that is not a prime factor of $p^{q^{n-1}}-1$, allowing us to conclude.
\end{proof}

	The categories constructed in \Cref{cor:charp} and \Cref{cor:subring} are stable categories that do not arise from additive categories. Moreover, they are not idempotent-complete, and the dual functor is the identity. Therefore we ask:

	\begin{qst}\label{qst:add}
		Given $R$ a commutative ring, when does there exist an additive rigid symmetric monoidal category $C$ with $K_0(C)\cong R$?
	\end{qst}
	
	\begin{qst}\label{qst:ide}
		Given $R$ a commutative ring, when does there exist an idempotent-complete rigid symmetric monoidal stable category $C$ with $K_0(C)\cong R$?
	\end{qst}

	\begin{qst}\label{qst:dual}
		Given $R$ a commutative ring with an involution, when does there exist a rigid symmetric monoidal stable category $C$ such that $K_0(C)$ acted on by the dual functor is equivalent to $R$?
	\end{qst}
%	We finally ask another variant of the question, where $C$ is asked to have a bounded $t$-structure. This would give a notion of homotopy group for the categorified invariant, which takes values in an abelian category.
%	
%	\begin{qst}\label{qst:tstr}
%		Given a subring $R \subset \QQ$ does there exist a stable idempotent-complete (rigid) symmetric monoidal category $C$ with $K_0(C)$ a field of characteristic zero and bounded $t$-structure compatible with the symmetric monoidal symmetric monoidal structure?
%	\end{qst}

\section{Universal and functorial examples}

	Next, we show it is possible to produce categories whose $K_0$ admit maps from a ring $R$ equipped with universal witnesses of relations in $K_0$. We then combine this with the results of the previous theorem to prove \Cref{thm:functorial} and \Cref{thm:functorialmon}. Given a symmetric monoidal stable category $C$, let $\Alg_{\EE_n}(\Mod(C))$ denote the category of $\EE_n$-monoidal $C$-linear categories\footnote{By an $\EE_n$-monoidal $C$-linear category we mean an $\EE_n$-algebra object (see \cite[Definition 5.1.0.4]{HA} and \cite[Definition 2.1.3.1]{HA}) in the symmetric monoidal category of modules over $C$ in the category of categories, such that the action map preserves finite colimits in both variables.}. For a discrete commutative ring $R$, we use $\Alg_{\EE_n}(\Mod(R)^{\heart})$ to denote the category of (discrete) associative $R$-algebras for $n=1$ and commutative $R$-algebras for $n>1$. For a category $E$ and object $e \in E$, we use $E_{e/}$ to denote the slice category of objects in $E$ equipped with a map from $e$.

	\begin{prop}\label{thm:versal}
		Let $1 \leq n\leq \infty$, $C$ be a symmetric monoidal stable category, and $C'$ a $\EE_n$-monoidal $C$-linear category. There is a functor $$D_{(-)}:\Alg_{\EE_n}(\Mod(K_0(C))^{\heart})_{K_0(C')/} \to \Alg_{\EE_n}(\Mod(C))_{C'/}$$ and a natural transformation
		$$\eta_{(R)}:R \to K_0(D_{R})$$
		such that $\eta_{R}$ makes $D_{R}$ a versal\footnote{or weakly initial} object in $\Alg_{\EE_n}(\Mod(C))_{C'/}$ equipped with a map  $R \to K_0(D_{R})$ in $\Alg_{\EE_n}(\Mod(K_0(C))^{\heart})_{K_0(C')/}$. The functor $D_{(-)}$ functorially depends on $C,C'$, and a choice of representative of each $K_0$-class of $K_0(C')$, and preserves filtered colimits with respect to all parameters. We can moreover require that $D_{(-)}$ take values and have a versal property instead in the subcategory of rigid categories.
		%such that $K_0(C') \to R \dasharrow K_0(D)$. $D$ is functorially constructed from the data of $C,C'$, and a presentation of $R$ as an associative $K_0(C)$-algebra under $K_0(C)$.
%		making the diagram below commute.
%		
%		\begin{center}
%			\begin{tikzcd}
%				K_0(C)\ar[r]\ar[dr] &R \ar[d]\\
%			& K_0(D)
%			\end{tikzcd}
%		\end{center}
	\end{prop}

\begin{proof}
%	The inclusion of rigid $\EE_n$-monoidal categories into $\EE_n$-monoidal categories admits a left adjoint, the rigidification functor. To see this, one can adjoin left duals of objects, i.e for each object $c$, one adjoins an object $c^{*}$
%	We will prove the proposition simultaneously without the rigid assumption, as the result with the rigid assumption follows by composing $D_{(-)}$ with the rigidification functor that forces an $\EE_n$-monoidal category to be rigid.
	We will prove the proposition simultaneously with and without the rigid assumption, using (rigid) to indicate which operations must be done as rigid categories with the rigid assumption.
	
	For each class in $x \in K_0(C')$, let us fix a choice of representative of the $K_0$-class of $x$, which we call $O_x$. Let us fix $R \in \Alg_{\EE_n}(\Mod(K_0(C))^{\heart})_{K_0(C')/}$. We choose a presentation of $R$ as an associative $K_0(C)$-algebra under $C'$ as follows:
	Let each $r \in R$ itself be the set of generators. To write the set of relations consider all formal expression of the form $\sum_{i=1}^n\prod_{j=1}^{m_i} x_{ij}$, where $x_{ij}$ is an element of the disjoint union $R\coprod K_0(C')$. For each such data, we get a relation if the relation $\sum_{i=1}^n\prod_{j=1}^{m_i} x_{ij}=0 \in R$ holds where for elements of $K_0(C')$, we consider their image in $R$.
	
	Now we let $D'_R$ be the free (rigid) $\EE_n$-monoidal $C$-linear category under $C'$ equipped with objects $X_{r}$ for each generator $r \in K_0(C')$. These objects give a natural map $\eta_{R}':K_0(C')\{R\} \to K_0(D'_R)$ in $\Alg_{\EE_n}(\Mod(K_0(C))^{\heart})_{K_0(C')/}$.
	
	Next, for each relation $\sum_{i=1}^n\prod_{j=1}^{m_i} x_{ij}$, consider the object $O = \oplus_{i=1}^n(\otimes_{j=1}^{m_i} P_{ij})$ where $P_{ij}$ is $O_{x_{ij}}$ whenever $x_{ij} \in K_0(C')$ and is $X_{x_{ij}}$ whenever $x_{ij} \in R$. Note that we consider the operations $\oplus$ and $\otimes$ as only giving monoidal structures (as opposed to symmetric monoidal and $\EE_n$-monoidal) so that $O$ is an object defined up to a contractible space of choices from the data defining the relation.
%	Let us fix a presentation of $R$ as an associative $K_0(C)$-algebra under $K_0(C')$, with generators $x_\alpha$ and relations $r_{\beta} \in K_0(C')\{x_{\alpha}\}$. We first define $D'$ to be the free (rigid) $\EE_n$-monoidal $C$-linear category under $C'$ equipped with objects $X_{\alpha}$. These objects give a natural map of associative algebras $K_0(C')\{x_{\alpha}\} \to K_0(D')$.

	For each relation, we freely adjoin to $D'_R$ as a (rigid) $\EE_n$-monoidal $C$-linear category objects $Y,Z$, maps
	
	\begin{center}
		\begin{tikzcd}[row sep=small]
			Y \ar[r,"f"] & O\oplus Z &
			Y \ar[r,"g"] & Z
		\end{tikzcd}
	\end{center}
	and an isomorphism $\cof(f) \cong \cof(g)$. Let $D_R$ be the object in $\Alg_{\EE_n}(\Mod(C))_{C'/}$ constructed via these operations.
	
	The fact that $\cof(f) \cong \cof(g)$ implies that $[O] +[Z] - [Y] = [Z]-[Y]$, i.e $[O] = 0$. Thus the composite $K_0(C')\{R\} \xrightarrow{\eta_R'} K_0(D'_R) \to K_0(D_R)$ factors through $R$, so we obtain the natural transformation $\eta_R$ as this factorization.
	
	It remains to show the versal property of $D_R$. To do this, let $E$ be a (rigid) $\EE_n$-monoidal $C$-linear category under $C'$ with a factorization $K_0(C') \to R \to K_0(E)$. We must then produce a map $D_R \to E$ in $\Alg_{\EE_n}(\Mod(C))_{C'/}$ compatible with this factorization.
	
	First, observe that by choosing representatives of the $K_0$-classes of each element of $R$, we obtain a map $D'_R \to E$ by sending the object $X_r, r \in R$ to this object. It suffices to show then that this functor $F:D'_R \to E$ admits a factorization through $D_R$. To do this, we use the following lemma:
%	Next we versally enforce the relations $r_{\beta}$. Each operation done below is done freely as a (rigid) $\EE_n$-monoidal $C$-linear category. For each relation $r_{\beta}$, let $R_{\beta}$ be the space of objects of $D'$ whose $K_0$ class is the image of the relation $r_{\beta}$ in the map $K_0(C')\{x_{\alpha}\} \to K_0(D')$. We have an inclusion functor $R_{\beta} \hookrightarrow D'$, and use $R_{\beta}(a)$ to denote the associated object of $D'$ for a point $a \in R_{\beta}$. We now adjoin an $R_{\beta}$-indexed families of objects $Y_{\beta}(a), W_{\beta}(a),Z_{\beta}(a)$ for $a \in R_{\beta}$, as well as families of maps $f_{\beta}(a), g_{\beta}(a)$
%	
%	\begin{center}
%	\begin{tikzcd}[row sep=small]
%		Y_{\beta}(a) \ar[r,"f_{\beta}(a)"] & R_{\beta}(a)\oplus Z_{\beta}(a)\\
%		Y_{\beta}(a) \ar[r,"g_{\beta}(a)"] & Z_{\beta}(a)
%	\end{tikzcd}
%\end{center}
%	as well as an $R_{\beta}$-indexed family of isomorphisms $\cof(f_{\beta}(a)) \cong \cof(g_{\beta}(a))$. Let $D$ be the resulting (rigid) $\EE_n$-monoidal $C$-linear category obtained from $D'$ via these operations. 
%	
%	The fact that $\cof(f_{\beta}(a)) \cong \cof(g_{\beta}(a))$ gives the relations
%	$$[Y_{\beta}(a)] - [R_{\beta}(a) \oplus Z_{\beta}(a)] = [Y_{\beta}(a)] - [Z_{\beta}(a)]$$ in $K_0(D')$, which is equivalent to the relation $r_{\beta} = 0$ in $K_0$. It then follows that the map $K_0(C')\{x_{\alpha}\} \to K_0(D') \to K_0(D)$ factors through $R$.
%	
%	To see that $D'$ is versal with respect to this property, we use the following lemma:
	
	\begin{lem}[Heller's criterion {\cite[Lemma 2.12]{kasprowski2019algebraic}}]\label{lem:k0crit}
		Let $C$ be a stable category. For any relation $[A]=[B]$ in $K_0(C)$, there exist cofibre sequences of the form
		\begin{center}
			\begin{tikzcd}[row sep=small]
				Y\ar[r,"f"] &A\oplus Z \ar[r] & J\\
				Y\ar[r,"g"] &B\oplus Z \ar[r] & J\\
			\end{tikzcd}
		\end{center}
	\end{lem}
	
	For a given relation $\sum_{i=1}^n\prod_{j=1}^{m_i} x_{ij}$, by applying \Cref{lem:k0crit} to $[F(O)] = [0]$ in $E$, we obtain the data as in the lemma. For each relation, by sending the $Y,Z,f,g$ used to form $D_R$ to these objects and maps, and by choosing an isomorphism between the $J$s appearing in the cofibre sequences, we obtain the desired factorization through $D_R$.	
%	Let $D''$ be a $C$-linear $\EE_n$-monoidal category under $C'$ with a factorization $K_0(C') \to R \to K_0(D'')$. By choosing objects representing the image of $x_\alpha$, we obtain a $C$-linear $\EE_n$-monoidal factorization $C' \to D \xrightarrow{F} D''$. For each relation $r_{\beta}$ and connected component of $R_{\beta}$, choose a point $a$. We then apply \Cref{lem:k0crit} to the relation $[R_{\beta}(a)] = [0]$ to obtain the data $H,I,J,f,g$ with $A = R_{\beta}(a)$ and $B = 0$. 
%	
%	Then for any point in the connected component of $a$ in $R_{\beta}$,
%	On that connected component of $R_{\beta}$, we can send $Y_{\beta}(a)$ to $H$, $Z_{\beta}(a)$ to $I$, $f_{\beta}(a)$ to $f$, and $g_{\beta}(a)$ to $g$. This produces a $C$-linear $\EE_n$-monoidal functor $D \to D''$ under $C'$ compatible with the factorization on the level of $K_0$.
\end{proof}

\begin{rmk}
	There is a version of \Cref{thm:versal} for additive categories: one need merely replace the application of \Cref{lem:k0crit} with the more basic fact that in an additive category, the relation $[A] = [B]$ holds iff there is some $Y$ such that $A \oplus Y \cong B\oplus Y$.%In another direction, the above proposition generalizes to arbitrary operads $\cO$ without change: one considers both the categories $C$ and $K_0$ as $\cO$-algebras and work with presentations of $K_0(C)$ as such.
\end{rmk}

Now we combine \Cref{cor:subring}, \Cref{cor:charp}, and \Cref{thm:versal} to prove \Cref{thm:functorial} and \Cref{thm:functorialmon}, whose statements we recall for convenience.

	\begin{customthm}{A}
	There is a filtered colimit preserving functor $C^{\infty}_{(-)}:\Ring_{\red}^0 \to \ZCat^{\infty}_{\rig}$ and a natural isomorphism of rings $K_0(C^{\infty}_R) \cong R$.
\end{customthm}

\begin{customthm}{B}
	There is a filtered colimit preserving functor $C^2_{(-)}:\Ring_{\red} \to \ZCat^2_{\rig}$ and a natural isomorphism of rings $K_0(C^2_R) \cong R$.
\end{customthm}

\begin{proof}[Proof of \Cref{thm:functorial} and \Cref{thm:functorialmon}]
	We first prove \Cref{thm:functorial}, and then indicate the necessary changes to prove \Cref{thm:functorialmon}. Letting $C = C' = \Mod(\ZZ)^{\omega}, n=\infty$, choosing any representatives of $K_0(C')$, and applying \Cref{thm:versal}, we obtain a functor $D_{(-)}$, taking a discrete commutative ring $R$ to $D_R$, a $\ZZ$-linear rigid symmetric monoidal category with a natural map $R \to K_0(D_R)$. The category $D_R$ comes functorially equipped with objects $X_r, r \in R$ (see the proof of \Cref{thm:versal}) whose $K_0$-class is the image of $r$. 
	
	We consider the ideal $I_R$ of $K_0(D_R)$ generated by $r- r^*$, where $r^*$ is the class of the dual of $r$. We apply \Cref{thm:versal} again with $C = \Mod(\ZZ)^{\omega}, C' = D_R, n=\infty$ and $X_r$ as our choice of representatives to obtain a category $D^{'}_R$ versally equipped with a map $K_0(D_R)/I_R \to K_0(D^{'}_R)$. Note that the image of the composite $R \to K_0(D_R)/I_R \to K_0(D^{'}_R)$ by construction has the property that it is fixed by the dual functor, so the subcategory of $D^{'}_R$ consisting of objects whose $K_0$-class is in this image is a rigid symmetric monoidal subcategory, which we define to be $C_{R}^{\infty}$. 
	
	It remains to show that the natural surjection $R \to K_0(C_R^{\infty})$ is an isomorphism, which is equivalent to the claim that $R \to K_0(D_R^{'})$ is injective. We now assume that each minimal prime of $R$ is characteristic $0$. Let $C$ be a $\ZZ$-linear rigid symmetric monoidal category as in \Cref{cor:subring} with the property that $K_0(C) \cong R$ and the action of the dual functor on $K_0$ is trivial. The fact that $K_0(C)\cong R$ by versality gives the existence of a symmetric monoidal $\ZZ$-linear functor $D_R \to C$. Because the action of the dual functor on $C$ is trivial, versality of $D'_R$ further allows us to extend this to a symmetric monoidal functor $D'_R \to C$. It then follows that the composite $R \to K_0(D'_R) \to K_0(C)$ is an isomorphism, so $R \to K_0(D'_R)$ is injective as desired.
	
	To prove \Cref{thm:functorialmon}, we run the same proof, except changing $n$ from $\infty$ to $2$, and replacing the use of \Cref{cor:subring} with \Cref{cor:charp}.
%	To prove \Cref{thm:functorialmon}, we run the same proof, with the following changes. We change $n$ from $\infty$ to $1$, and replace the use of \Cref{cor:subring} with \Cref{cor:charp}, and changing $I_R$ to be the ideal generated by the differences of $r \in K_0(D_R)$ with both the left and right duals or $r$.
\end{proof}

\begin{rmk}
	One can use \Cref{thm:versal} to make an `obstruction theory' for constructing \textit{idempotent-complete} categories with a specified $K_0$. We briefly indicate how this works for a symmetric monoidal categories.
	
	Given $C_0$ an idempotent-complete symmetric monoidal rigid stable category and a map $K_0(C_0) \to R$, we can by \Cref{thm:versal} construct a versal idempotent-complete rigid symmetric monoidal functor $C_0 \to C_1$ with a factorization $K_0(C_0) \to R \to K_0(C_1)$. The first obstruction to constructing an idempotent-complete $C_0$-linear category with $K_0$ isomorphic to $R$ is whether the map $R \to K_0(C_1)$ admits a retraction.
	
	If a retraction exists, we can choose a retraction, and form $C_2$ as a versal idempotent-complete $C_1$-linear category with a factorization $K_0(C_1) \to R \to K_0(C_2)$. The second obstruction is whether the map $R \to K_0(C_2)$ admits a retraction. 
	
	One can keep going, and if all obstructions vanish, one can produce for each $i$ a category $C_i$, and the filtered colimit of $C_i$ is an idempotent-complete $C_0$-linear category with the desired $K_0$. On the other hand, versality shows that if an idempotent-complete $C_0$-linear category exists with the desired $K_0$, then choices of retractions can be made so that all obstructions vanish.
\end{rmk}

\begin{rmk}
	Using \Cref{thm:versal} as in \Cref{thm:functorial}, one can reduce the problem of categorifying arbitrary commutative rings to the case of finite commutative rings. This is because the versal category $C_R$ with a surjective map $R \to K_0(C_R)$ preserves filtered colimits in the variable $R$, so to show the natural map $R \to K_0(C_R)$ is an isomorphism, one can assume $R$ is finitely presented. In this case, $R$ is a Noetherian Jacobson ring, so embeds into a product of finite rings (namely the quotients by powers of all maximal ideals), so it suffices to solve the problem for those. It is not clear if there is a `small' collection of finite rings that generate the rest under subrings, products and ultraproducts, but rings such as $\ZZ/p^n$ and $\FF_q[x_1\dots,x_n]/(x_i^{m_i}) $ can be categorified via methods similar to those presented in this paper, and seem to generate a lot of finite rings.
\end{rmk}

	%\nocite{}
	\printbibliography
	\Addresses
\end{document}